\documentclass[final]{amsart}

\usepackage{
  amsmath,
  amssymb,
  amsthm,
  hyperref,
  paralist,
  thmtools,
  tikz,
  tikz-cd,
  pgfplots,
  verbatim
}
\usepackage[margin=1.5in]{geometry}
\usetikzlibrary{arrows, automata, backgrounds, calc, positioning, shapes, shapes.geometric, decorations.pathreplacing}
\usepackage[nocompress]{cite}
\usepackage[textsize=scriptsize]{todonotes}
\presetkeys{todonotes}{inline}{}

\usepackage[draft,inline]{showlabels}

\makeatletter
\providecommand\@dotsep{5}
\makeatother

\renewcommand{\k}{\mathbf{k}}

\newcommand{\calC}{\mathcal C}

\newcommand{\from}{\colon}
\newcommand{\heart}{\heartsuit}

\DeclareMathOperator{\Aut}{Aut}

\DeclareMathOperator{\Hom}{Hom}

\DeclareMathOperator{\Stab}{Stab}

\DeclareMathOperator{\coker}{coker}

\renewcommand{\top}{\operatorname{top}}
\renewcommand{\bot}{\operatorname{bot}}

\declaretheorem[parent=section]{theorem}
\declaretheorem[sibling=theorem]{lemma}
\declaretheorem[sibling=theorem]{corollary}

\declaretheorem[sibling=theorem]{proposition}
\declaretheorem[sibling=theorem, style=remark]{remark}

 \DeclareMathOperator{\HHom}{Hom}

\title{Spherical objects and stability conditions on 2-Calabi--Yau quiver categories}
\author{Asilata Bapat \and Anand Deopurkar \and Anthony M. Licata}
\begin{document}

\begin{abstract}
  Consider a 2-Calabi--Yau triangulated category with a Bridgeland stability condition.
  We devise an effective procedure to reduce the phase spread of an object by applying spherical twists.
  Using this, we give new proofs of the following theorems for 2-Calabi--Yau categories associated to ADE quivers:
  \begin{inparaenum}
  \item all spherical objects lie in a single orbit of the braid group, and
  \item the space of Bridgeland stability conditions is connected.
  \end{inparaenum}  
\end{abstract}

\maketitle

\section{Introduction}
Suppose we have a category $\mathcal C$ with a measure of complexity of its objects.
To what extent can we simplify an object by applying suitable auto-equivalences of $\mathcal C$?
In this paper, we explore this question for $2$-Calabi--Yau triangulated categories with respect to a measure of complexity provided by a Bridgeland stability condition. 
We focus on the question of simplifying spherical objects by applying a sequence of auto-equivalences known as spherical twists.

Spherical objects play an important role in Calabi--Yau categories.
Recall that an object of an $n$-Calabi--Yau category is spherical if its endomorphism ring is isomorphic to the cohomology ring of the $n$-sphere---this is the simplest possible endomorphism ring, given the $n$-Calabi--Yau condition.
Every spherical object gives an auto-equivalence of the category, called a spherical twist~\cite{sei.tho:01}.

Spherical objects interact wonderfully with Bridgeland stability conditions.
For example, the Harder--Narasimhan factors of any spherical object with respect to a Bridgeland stability condition are themselves spherical (see, e.g.~\cite[Corollary 2.3]{huy:12}).
In some cases (for example, for derived categories of K3 surfaces), a stability condition is determined by its behaviour on spherical objects (see \cite{huy:12}).

Our first set of results (\autoref{sec:impromevent}) discusses how to simplify arbitrary objects of a 2-Calabi--Yau triangulated category $\mathcal C$ by applying twists by spherical objects that are stable with respect to a given stability condition.
We show that under suitable hypotheses, we can strictly decrease the phase spread of an object by applying appropriate positive and negative twists.
If the set of possible spreads is a discrete subset of $\mathbb R$, then this procedure must terminate.
As a result, we get a particularly simple representative of the orbit of the object under the spherical twist group.

The phase improvement results raise the question of classifying the stable spherical objects with respect to a given stability condition.
Our next set of results (\autoref{sec:sphericalstable}) achieves this for the 2-Calabi--Yau category $\mathcal C_\Gamma$ associated to a quiver $\Gamma$.
We briefly describe the classification.
The Grothendieck group $K\left(\mathcal C_\Gamma\right)$, together with the $\hom$ pairing is naturally isomorphic to the root lattice of $\Gamma$.
The spherical objects of $C_\Gamma$ map to real roots of $K\left(C_\Gamma\right)$.
The action of the Coxeter group $W_\Gamma$ on the root lattice lifts to an action of the Artin--Tits braid group $B_\Gamma$ on $\mathcal C_\Gamma$.
The standard generators of $B_\Gamma$ act on $\mathcal C_\Gamma$ by spherical twists.

Let $\tau$ be a generic stability condition on $\mathcal C_\Gamma$ whose heart is the standard heart.
Given a root $w \in K\left( C_\Gamma \right)$, there are multiple (in fact, infinitely many, even modulo triangulated shifts) spherical objects with class $w$.
However, only one of them (up to triangulated shift) is $\tau$-stable.
We give an explicit combinatorial algorithm to construct this object.
More precisely, suppose we have a minimal expression for $w$ in terms of simple reflections applied to a simple root $v$, say
\[ w = s_{v_n} \cdots s_{v_1}v.\]
Let $P_v$ be the simple object of the standard heart of class $v$.
We prove that the unique $\tau$-stable object $P_w$ of class $w$ can be expressed as a lift of the above expression:
\[ P_w = \sigma_{v_n}^{\pm 1} \cdots \sigma_{v_1}^{\pm 1} P_v,\]
and moreover we explicitly identify the exponents.
The choice of signs in the exponents is governed by the interaction of the root sequence for the expression for $w$ with the central charge.

In the last section (\autoref{sec:applications}), we give two applications of the simplification procedure and of the classification of spherical stable objects mentioned above.
The first application is a new proof of the following theorem.
\begin{theorem}\label{thm:spherical}
  Let $\Gamma$ be a quiver of type $A_n$, $D_n$, or $E_6$, $E_7$, $E_8$.
  The spherical objects of $\mathcal C_\Gamma$ lie in the $B_\Gamma$ orbit of the simple objects of the standard heart.
\end{theorem}
In the main text, \autoref{thm:spherical} is \autoref{cor:spherical}.
A proof of this theorem in type $A$ appears in \cite{ish.ueh:05,ish.ued.ueh:10} and may also follow for all $ADE$ types from the ideas of \cite{ada.miz.yan:19}.

The second application is a new proof of the following theorem.
\begin{theorem}\label{thm:stability}
  Let $\Gamma$ be a quiver of type $A_n$, $D_n$, or $E_6$, $E_7$, $E_8$.
  Any stability condition $\tau \in \Stab(\mathcal C_\Gamma)$ is in the $B_\Gamma$ orbit of a standard stability condition.
  Furthermore, $\Stab(\mathcal C_\Gamma)$ is connected.
\end{theorem}
In the main text, \autoref{thm:stability} is \autoref{cor:stability}.
\autoref{thm:stability} has been proved in~\cite{ish.ueh:05,ish.ued.ueh:10} for type $A$, and in~\cite{ada.miz.yan:19} for types $A$, $D$, and $E$.

We highlight that the classification of spherical stable objects and the phase reduction procedure holds for any quiver $\Gamma$.
We require the finite type assumption to guarantee the termination of the procedure.
It is likely that termination works more generally (we have not found any counter-examples).
We hope to address this question in the future.

\subsection*{Acknowledgements}
We are grateful to Fabian Haiden, Ailsa Keating, Alexander Polishchuk, and Catharina Stroppel for discussions that prompted us to write this paper.
We are indebted to Tom Bridgeland for his encouragement and for discussions and suggestions about phase improvement.

\section{Background on stability conditions and spherical objects}\label{sec:background}
Let $\mathcal C$ be an arbitrary 2-Calabi--Yau triangulated category.
We begin by recalling the notion of spherical objects, spherical twists, and bottom/top phases with respect to a stability condition.
Recall that an object $X$ of $\calC$ is spherical if we have an isomorphism of $\k$-algebras
\[ \Hom^*(X,X) \cong H^*(S^2, \k).\]
In what follows, we write $\Hom(X,Y)$ for the graded vector space $\Hom^*(X,Y)$, namely the direct sum $\bigoplus_n \Hom^0(X, Y[n])$.
We have a pairing on the Grothendieck group $K(\mathcal C)$, defined as
\[
  \langle X,Y \rangle = \sum_{i=-\infty}^\infty \dim \Hom^i(X,Y).
\]
We call this the \emph{$\Hom$ pairing}.

A spherical object $X$ defines a triangulated auto-equivalence of $\calC$, called the spherical twist in $X$, denoted by $\sigma_X$.
An object $Y$ and its spherical twist $\sigma_X(Y)$ are related by the exact triangle
\begin{equation}\label{eqn:positive-twist}
  \HHom(X,Y) \otimes X \to Y \to \sigma_X(Y) \xrightarrow{+1},
\end{equation}
where the map $\HHom(X,Y) \otimes X \to Y$ is the evaluation map.
Likewise, $Y$ and the inverse twist $\sigma^{-1}_X(Y)$ are related by the exact triangle
\begin{equation}\label{eqn:negative-twist}
  \sigma^{-1}_X(Y) \to Y \to  X \otimes \HHom(Y,X)^\vee \xrightarrow{+1},
\end{equation}
where the map $Y \to X \otimes \HHom(Y,X)^\vee$ is obtained from the evaluation adjoint to the evaluation map.

Let $\tau$ be a stability condition on $\calC$.
Denote by $\phi_\tau(X)$ or simply $\phi(X)$ the phase of a semi-stable object $X$.
An arbitrary object $X$ has a unique Harder--Narasimhan (HN) filtration
\[
  0 = X_0 \to X_1 \to \dots \to X_n = X,
\]
where the sub-quotients $Z_i$, defined by triangles $X_{i-1} \to X_i \to Z_i \xrightarrow{+1}$, are $\tau$-semistable and satisfy
\[ \phi(Z_1) > \cdots > \phi(Z_n).\]
With the notation above, set
\begin{align*}
  \phi^+(X) = \phi(Z_1), &\quad \phi^-(X) = \phi(Z_n),\\
  X^{\top} = Z_1, &\text{ and }  X^{\bot} = Z_n.
\end{align*}
We call $\phi^+(X)$ (resp. $\phi^-(X)$) the \emph{top} (resp. \emph{bottom}) phase of $X$.

Recall that the definition of a stability condition implies that $\Hom(X, Y) = 0$ if $X$ an $Y$ are semistable with $\phi(X) > \phi(Y)$.
More generally, an easy induction shows that $\Hom(X, Y) = 0$ when $\phi^-(X) > \phi^+(Y)$.

\section{Phase impromevent using spherical twists}\label{sec:impromevent}
The goal of this section is to prove that by applying suitable spherical twists, we can predictably increase/decrease the bottom/top phase of an object.
Throughout, fix an arbitrary $\k$-linear 2-Calabi--Yau triangulated category $\calC$ and a stability condition $\tau$ on $\calC$.

The following is standard.
\begin{lemma}[Sandwich lemma]
  \label{prop:sandwich}
  Let $X \to Y \to Z \xrightarrow{+1}$ be an exact triangle.
  Then
  \begin{align*}
    \phi^-(Y) &\geq \min \{\phi^-(X), \phi^-(Z)\}, \text{ and }\\
    \phi^+(Y) &\leq \max \{\phi^+(X), \phi^+(Z)\}.
  \end{align*}
\end{lemma}
\begin{proof}
  We prove the first inequality; the second is analogous.
  We have a nonzero map $Y \to Y^{\bot}$.
  As a result, there is either a nonzero map $X \to Y^{\bot}$ or a nonzero map $Z \to Y^{\bot}$.
  This shows that $\phi(Y^{\bot}) \geq \phi^-(X)$ or $\phi(Y^{\bot}) \geq \phi^-(Z)$.
  Equivalently, $\phi^-(Y) \geq \min\{\phi^-(X),\phi^-(Z)\}$.
\end{proof}

We now investigate the effect of applying suitable spherical twists on the bottom and the top phase.
\begin{proposition}
  \label{prop:improvement1}
  Let $X$ be a spherical stable object of $\calC$ such that $X$ is the unique stable object of its phase.
  Let $Y$ be any object of $\calC$.
  We have the following.
  \begin{enumerate}
  \item If $\phi(X) \leq \phi^-(Y)$, then $\phi(X) < \phi^-(\sigma_X^{-1}(Y))$.
  \item If $\phi^+(Y) \leq \phi(X)$, then $\phi^+(\sigma_X(Y)) < \phi(X)$.
  \end{enumerate}
\end{proposition}
\begin{proof}
  We prove the first inequality; the second is analogous.
  The key is the exact triangle
  \begin{equation}\label{eq:rotated-neg-twist-triangle}
    X \otimes \HHom(Y,X)^\vee[-1] \to \sigma_X^{-1}(Y) \to Y \xrightarrow{+1},
  \end{equation}
  obtained by rotating the triangle in \eqref{eqn:negative-twist} in \autoref{sec:background}.
  
  We first prove the proposition assuming the strict inequality
  \begin{equation}\label{eqn:phase1}
    \phi(X) < \phi^-(Y).
  \end{equation}
  Let $V = \HHom(Y,X)^\vee[-1]$.
  Since $V$ is a complex of vector spaces, it is quasi-isomorphic to a direct sum of shifts of copies of $\k$.
  The inequality \eqref{eqn:phase1} implies that $\Hom^i(Y,X) = 0$ for $i \leq 0$.
  Therefore, $V$ is a direct sum of copies of $\k[j]$ for $j \geq 0$.
  As a result, we have
  \begin{equation}\label{eqn:phase2}
    \phi(X) \leq \phi^-(X \otimes V).
  \end{equation}
  Thanks to \eqref{eqn:phase1} and \eqref{eqn:phase2}, we can apply the sandwich lemma (\autoref{prop:sandwich}) to the key triangle \eqref{eq:rotated-neg-twist-triangle} to get
  \[\phi(X) \leq \phi^-(\sigma_X^{-1}(Y)).\]
  
  To show that the inequality is strict, it suffices to show that $\sigma_X^{-1}Y$ does not have a nonzero map to any stable object of phase $\phi(X)$.
  By our assumption, the only such stable object is $X$ itself.
  Consider a map $f \from \sigma_X^{-1}(Y) \to X$.
  Applying $\sigma_X$ gives a map $\sigma_X(f) \from Y \to X[-1]$.
  Since $\phi(X[-1]) < \phi(X) < \phi^-(Y)$, the map $\sigma_X(f)$ must be zero.
  Therefore $f$ is zero.
  The proof is thus complete, assuming $\phi(X) < \phi^-(Y)$.

  We now treat the case $\phi(X) = \phi^-(Y)$.
  By our assumption on $X$, in this case, $Y^{\bot}$ must be a direct sum of copies of $X$.
  This means that we have an exact triangle
  \begin{equation}\label{eqn:Z-triangle}
    Z \to Y \to X^{\oplus n} \xrightarrow{+1},
  \end{equation}
  where $\phi(X) < \phi^-(Z)$.
  The previous argument now applies to $Z$, and we get $\phi(X) < \phi^{-1}\left(\sigma_X^{-1}Z\right)$.
  Applying $\sigma_X^{-1}$ to the triangle in \eqref{eqn:Z-triangle} gives the triangle
  \begin{equation}\label{eqn:sZ-triangle}
    \sigma_X^{-1}(Z) \to \sigma_X^{-1}(Y) \to X[1]^{\oplus n} \xrightarrow{+1}.
  \end{equation}
  By applying the sandwich lemma (\autoref{prop:sandwich}) to \eqref{eqn:sZ-triangle}, we conclude that $\phi(X) < \phi^-(\sigma_X^{-1}(Y))$.
\end{proof}

The following shows that the improvement on one end achieved by \autoref{prop:improvement1} does not cause a deterioration at the other end, assuming that the object $Y$ has a sufficiently large phase spread.
\begin{proposition}
  \label{prop:improvement2}
  Let $X$ be a spherical stable object of $\calC$.
  Assume that $X$ is the unique stable object of its phase.
  Let $Y$ be any object of $\calC$ such that $\Hom^{i}(Y,Y) = 0$ for any $i < 0$ and $\phi^+(Y)-\phi^-(Y) \geq 1$.
  The following hold.
  \begin{enumerate}
  \item If $\phi^-(Y) = \phi(X)$, then $\phi^+(Y) \geq \phi^+(\sigma_X^{-1}(Y))$.
  \item If $\phi^+(Y) = \phi(X)$, then $\phi^-(Y) \leq \phi^-(\sigma_X(Y))$.
  \end{enumerate}
\end{proposition}
\begin{proof}
  We prove the first statement; the second is analogous.
  Again, the key is the exact triangle
  \begin{equation}\label{eqn:key2}
    X \otimes \HHom(Y,X)^\vee[-1] \to \sigma_X^{-1}(Y) \to Y  \xrightarrow{+1},
  \end{equation}
  obtained by rotating the triangle \eqref{eqn:negative-twist} in \autoref{sec:background}.
  Let $V = \HHom(Y,X)^\vee[-1]$.
  By the sandwich lemma (\autoref{prop:sandwich}), it suffices to show that
  \begin{equation}\label{eqn:to-show}
    \phi^+(Y) \geq \phi^+(X\otimes V).
  \end{equation}

  Let us compute $\phi^+(X \otimes V)$.
  Let $\ell$ be the largest integer such that $\Hom^\ell(Y,X) \neq 0$.
  Then $V$ is a direct sum of copies of $\k[j]$ with $j \leq \ell - 1$, and including at least one copy of $\k[\ell - 1]$.
  Therefore, we get
  \begin{equation}\label{eqn:lowest-hom}
    \phi^+(X \otimes V) = \phi(X[\ell - 1]).
  \end{equation}
  Thus, showing \eqref{eqn:to-show} is equivalent to showing 
  \[ \phi^+(Y) \geq \phi(X[\ell-1]).\]
  
  First suppose $\ell \leq 2$.
  Then we have
  \[ \phi^-(Y) + 1 = \phi(X[1]) \geq \phi(X[\ell - 1]).\]
  Since $\phi^+(Y) - \phi^-(Y) \geq 1$, we conclude that
  \[ \phi^+(Y) \geq \phi(X[\ell - 1]), \]
  as desired.
  
  Next suppose $\ell > 2$.
  Then the 2-Calabi--Yau property implies that
  \[
    \Hom^{2-\ell}(X, Y) \neq 0.
  \]
  By our assumptions on $X$ and $Y$, the object $Y^{\bot}$ is a direct sum of copies of $X$.
  Therefore, we also have
  \[ \Hom^{2 - \ell}(Y^{\bot}, Y) \neq 0.\]
  Define the object $K$ by the following exact triangle:
  \[ K \to Y \to Y^{\bot} \xrightarrow{+1}.\]
  Let $f \in \Hom^{2-\ell}(Y^{\bot}, Y)$ be a non-zero element.
  The composition of $Y \to Y^{\bot}$ with $f$ gives a map $Y \to Y[2-\ell]$.
  Since this is a map of negative degree from $Y$ to itself, it must be zero.
  Therefore $f$ factors as the composite of $Y^{\bot} \to K[1]$ with a (non-zero) map $g \from K[1] \to Y[2 - \ell]$
  \[
    \begin{tikzcd}
      Y \arrow{r} \arrow{dr}{0}& Y^{\bot} \arrow{r}\arrow{d}{f}& K[1] \arrow{r}{+1} \arrow{dl}{g}& {} \\
      & Y[2-\ell].& 
    \end{tikzcd}
  \]
  Since $g$ is non-zero, we get
  \[ \phi^+(Y[2-\ell]) \geq \phi^-(K[1]).\]
  By construction, we have
  \[ \phi^-(K) > \phi^-(Y) = \phi(X). \]
  By combining the last two inequalities, we see that
  \[ \phi^+(Y[2 - \ell]) > \phi(X[1]).\]
  Therefore, we get
  \[ \phi^+(Y) > \phi(X[\ell - 1]),\]
  as desired.
\end{proof}

The following is an analogue of \autoref{prop:improvement2} for $Y$ of small spread.
\begin{proposition}
  \label{prop:improvement3}
  Let $X$ be a spherical stable object of $\calC$.
  Assume that $X$ is the unique stable object of its phase.
  Let $Y$ be any object of $\calC$ such that $\phi^+(Y) - \phi^-(Y) < 1$.
  Assume, furthermore, that $\Hom^0(Y,Y) = \k$ and that $X$ is not a direct summand of $Y$.
  The following hold.
  \begin{enumerate}
  \item If $\phi^-(Y) = \phi(X)$, then $\phi^+(Y) \geq \phi^+(\sigma_X^{-1}(Y))$.
  \item If $\phi^+(Y) = \phi(X)$, then $\phi^-(Y) \leq \phi^-(\sigma_X(Y))$.
  \end{enumerate}
\end{proposition}
\begin{proof}
  We prove the first statement; the second is analogous.
  We begin as in the proof of \autoref{prop:improvement2}.
  Consider the triangle
  \begin{equation}\label{eqn:key3}
    X \otimes \HHom(Y,X)^\vee[-1] \to \sigma_X^{-1}(Y) \to Y  \xrightarrow{+1}.
  \end{equation}
  Let $V = \HHom(Y,X)^\vee [-1]$.
  It suffices to show that
  \[\phi^+(Y) \geq \phi^+(X \otimes V).\]
  Let $\ell$ be the largest integer such that $\Hom^\ell(Y, X) \neq 0$.
  Then we must show that
  \[ \phi^+(Y) \geq \phi(X[\ell - 1]).\]
  
  By the 2-Calabi--Yau property, we have $\Hom^\ell(Y,X) \cong \Hom^{2-\ell}(X,Y)^\vee$.
  Since
  \[\phi^+(Y) < \phi^-(Y) + 1 = \phi(X) + 1,\]
  there cannot be a non-zero map from $X$ to $Y[k]$ for $k < 0$.
  As a result, we must have $\ell \leq 2$.

  First suppose $\ell \leq 1$.
  Then we have
  \[ \phi^+(Y) \geq \phi^-(Y) = \phi(X) \geq \phi(X[\ell - 1]),\]
  as desired.

  We rule out $\ell = 2$.
  Let us show that if $\ell = 2$, then $X$ must be a direct summand of $Y$.
  
  Let $\mathcal P$ be the slicing defined by the stability condition $\tau$.
  Since $\phi^+(Y) - \phi^-(Y) < 1$ and $\phi(X) = \phi^-(Y)$, both $X$ and $Y$ lie in the abelian category $\mathcal P[\alpha, \alpha+1)$ for $\alpha = \phi^-(Y)$.
  By our assumptions, $Y^{\bot}$ is a direct sum of copies of $X$, say $Y^{\bot} = X^{\oplus n}$.

  If $\ell = 2$, we have a non-zero map $i \from X \to Y$.
  Consider the composite
  \begin{equation}\label{eqn:composite}
    X \xrightarrow{i} Y \to Y^{\bot}.
  \end{equation}
  We show that the composite is non-zero.
  Equivalently, we must show that $i$ does not factor through the kernel $K$ of $Y \to Y^{\bot}$.
  In fact, let us prove that there are no non-zero maps from $X$ to $K$.

  Since $\Hom^0(Y,Y) = \k$, every non-zero map from $Y$ to itself is an isomorphism.
  A non-zero map $X \to K$ gives a non-zero map $Y \to Y$ that is not an isomorphism, namely the composite
  \[Y \twoheadrightarrow Y^{\bot} = X^{\oplus n} \twoheadrightarrow X \to K \hookrightarrow Y.\]
  Therefore, there are no non-zero maps $X \to K$.

  Since we have $Y^{\bot} = X^{\oplus n}$ and the composite in \eqref{eqn:composite} is non-zero, there is a map $\pi \from Y \to X$ such that $\pi \circ i \from X \to X$ is non-zero.
  But $X$ is spherical, so $\pi \circ i$ must be an isomorphism.
  That is, $X$ is a direct summand of $Y$, as desired.
\end{proof}
\begin{remark}
  In \autoref{prop:improvement1}, \autoref{prop:improvement2}, \autoref{prop:improvement1}, suppose we know \emph{a priori} that the stable factors of \(Y\), and any spherical twist applied to \(Y\), are spherical.
  This is true, for example, if \(Y\) is itself a spherical object, or a direct sum of spherical objects (see \cite[Corollary~2.3]{huy:12}).
  Then we may weaken the uniqueness assumption on \(X\) to the following: \(X\) is the unique \emph{spherical} stable object of its phase.
\end{remark}

\section{Spherical stable objects in quiver categories}\label{sec:sphericalstable}
In this section we will consider the 2-Calabi--Yau category associated to an arbitrary quiver (not necessarily of finite type), as defined, e.g.,  in~\cite[Section 2]{bap.deo.lic:20*1}.

In the remainder of this section, take $\Gamma$ to be an arbitrary quiver, not necessarily of finite type.
Let $\calC$ be the 2-Calabi--Yau category associated to $\Gamma$ (see, e.g. \cite[Section 2]{bap.deo.lic:20*1}).

The Grothendieck group $K(\calC)$ with the $\Hom$ pairing is naturally identified with the root lattice of $\Gamma$.
Since $\dim(\Hom^*(X,X)) = 2$ for any spherical object $X$, its class $[X]$ in $K(\calC)$ is a real root.

Let $\tau$ be a stability condition on $\calC$ with central charge $Z \from K(\calC) \to \mathbb C$ and slicing $\mathcal P$.
\begin{proposition}\label{prop:unique-spherical-stable}
  Assume that $\tau$ is generic in the following sense: $Z$ maps distinct real roots to complex numbers of distinct phase.
  Suppose $X$ is a $\tau$-semistable spherical object.
  Then $X$ is $\tau$-stable, and it is the unique $\tau$-stable spherical object of phase $\phi = \phi(X)$.
\end{proposition}
\begin{proof}
  To show that $X$ is stable, we must show that $X$ is simple in the abelian category $\mathcal P(\phi)$.
  Let $S \subset X$ be a non-zero simple sub-object.
  By \cite[Corollary~2.3]{huy:12}, $S$ must be spherical.
  Then the class of $S$ in $K(\calC)$ is a real root.
  Since $Z(S)$ has the same argument as $Z(X)$, the genericity assumption on $\tau$ means that $X = S$ in $K(\calC)$.
  But then $X/S = 0$ in $K(\calC)$, and since $X/S$ is in a heart of $\tau$, this implies that $X = S$.

  Next, suppose $Y$ is another $\tau$-stable spherical object of the same phase as $X$.
  Again by the genericity of $\tau$, we have $X = Y$ in $K(\calC)$.
  But then we get
  \begin{align*}
    \langle X,Y \rangle & = \dim \Hom^0(X,Y) - \dim \Hom^1(X,Y) + \dim \Hom^2(X,Y) \\
                     & = \langle X, X\rangle = 2,
  \end{align*}
  which implies $\Hom^0(X,Y) \neq 0$ by the Calabi--Yau property.
  Since $X$ and $Y$ are simple objects of $\mathcal P(\phi)$, this forces $X \cong Y$.
\end{proof}

For every simple root $v$, we have a spherical object $P_v \in \mathcal C$.
Recall that the extension closure of these objects is the heart of a bounded $t$-structure on $\mathcal C$.
We call this the \emph{standard heart} and denote it $\heart$.
The objects $P_v$ corresponding to simple roots are simple in $\heart$.
We say that a stability condition is \emph{standard} if its \([0,1)\) heart is $\heart$.
We now give an effective construction of a $\tau$-stable spherical object of every possible class, for a generic standard stability condition $\tau$.

Let $w$ be an arbitrary positive root.
Write
\begin{equation}\label{eqn:wv}
  w = s_{v_n} \cdots s_{v_1} v,
\end{equation}
where $v$ is a simple root and $s_{v_1}, \dots, s_{v_n}$ are reflections in the simple roots $v_1, \dots, v_n$.
Set $v_0 = v$.
Associate to \eqref{eqn:wv} a sequence of roots $R_0, \dots, R_n$ defined by
\[
  R_i = s_{v_n} \cdots s_{v_{i+1}}(v_i).
\]
Note that $R_0 = w$ and if \eqref{eqn:wv} is a minimal expression for $w$, then all the roots in the root sequence are positive.

Let $\epsilon_1, \dots, \epsilon_n$ be $\pm 1$.
Consider an object $X$ of $\mathcal C$ defined by
\begin{equation}\label{eqn:lift}
  X =  \sigma^{\epsilon_n}_{v_{n}} \circ \dots \circ \sigma_{v_{1}}^{\epsilon_1} (P_{v}),
\end{equation}
where $\sigma_{v_i}$ is the spherical twist in $P_{v_i}$.
The $\epsilon$'s allow us to divide the root sequence into \emph{positive} and \emph{negative} sub-sequences, defined by
\begin{align*}
  R_+ &= \left( R_i \mid \epsilon_i = 1 \right), \\
  R_- &= \left( R_i \mid \epsilon_i = - 1 \right).\\
\end{align*}
We call $R_0$ the \emph{neutral root}.

Let $\tau$ be a stability condition on $\calC$ such that the $[0,1)$ heart of $\tau$ is the standard heart.
Let $Z$ be the central charge of $\tau$.
Let $\mathbb H \subset \mathcal C$ be the half-open upper half plane:
\[ \mathbb H = \{z \mid \Im(z) > 0\} \cup \mathbb R_{> 0}.\]
Let $\alpha = Z(R_0)$.
Since $R_0$ is a positive root, $Z(R_0)$ lies in $\mathbb H$.
It divides $\mathbb H$ into two pieces
\begin{align*}
  \mathbb H_+ &= \{z \mid \arg z > \arg \alpha\}, \\
  \mathbb H_- &= \{z \mid \arg z < \arg \alpha\},
\end{align*}
where $\arg$ is taken in $[0,\pi)$.

\autoref{fig:rootsequence} shows an example of the construction above for a stability condition on the $A_3$-category.

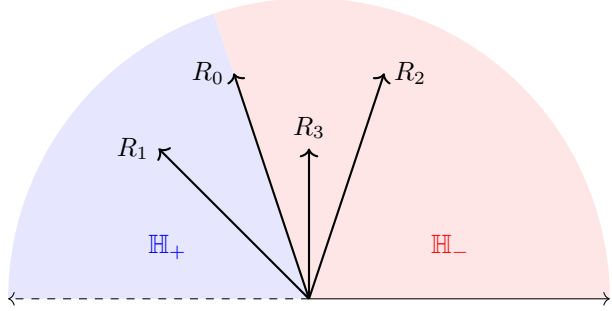
\begin{figure}
  \centering
  \begin{tikzpicture}[scale=2]
    \draw[->] (0, 0) -- (2,0);
    \draw[->, dashed] (0,0) -- (-2,0);
    \draw[->, thick] (0,0) --  (0,1) node [above] {$R_3$};
    \draw[->, thick] (0,0) --  (0.5,1.5) node [right] {$R_2$};
    \draw[->, thick] (0,0) --  (-1,1) node [left] {$R_1$};
    \draw[->, thick] (0,0) --  (-0.5,1.5) node [left] {$R_0$};
    \begin{scope}[on background layer]
      \fill[draw=none, color=white!90!red] (0:2.0) arc (0:108.5:2.0) -- (0,0);
      \fill[draw=none, color=white!90!blue] (108.5:2.0) arc (108.5:180:2.0) -- (0,0);
      \draw (20:1.0) node [color=red] {$\mathbb H_-$};
      \draw (160:1.0) node [color=blue] {$\mathbb H_+$};
    \end{scope}
  \end{tikzpicture}
  \caption{
    \protect
    Consider the $A_3$ quiver with simple roots $\alpha_i$ and simple reflections $s_i$.
    Consider the root sequence for  $w = s_2s_3s_1(\alpha_2)$.
    The central charge chosen for the diagram above maps $R_1$ to $\mathbb H_+$ and $R_2, R_3$  to $\mathbb H_-$.
    By \autoref{prop:unique-spherical-stable-construction}, the stable object of class $w$ is $\sigma^{-1}_2 \sigma^{-1}_3 \sigma_1(P_2)$.
  }\label{fig:rootsequence}
\end{figure}

Since all the roots in the root sequence $R$ of $X$ are positive, they are mapped to the upper half plane $\mathbb H$ by the central charge.
The semi-stability of $X$ depends on their position with respect to $R_0$.
\begin{proposition}\label{prop:unique-spherical-stable-construction}
  With the notation above, the object $X$ defined by \eqref{eqn:lift} is $\tau$-semistable if and only if $Z(R_+) \subset \mathbb H_+$ and $Z(R_-) \subset \mathbb H_-$.
\end{proposition}
Let $\tau$ be generic.
\autoref{prop:unique-spherical-stable-construction} gives an effective construction of the unique $\tau$-stable object of class $w = R_0$.
Indeed, we compute the root sequence $R$, look at its image in $Z(R) \subset \mathbb H$, and take $\epsilon_i = \pm 1$, depending on whether $Z(R_i)$ lies in $\mathbb H_\pm$.

We need some preparation to prove~\autoref{prop:unique-spherical-stable-construction}, including the definitions and basic properties of spherical twists from the beginning of \autoref{sec:impromevent}.
Let $\heart \subset \mathcal C$ be the standard heart.
Set
\[ K = K(\heart) = K(\mathcal C).\]
Denote by $[X]$, the class in $K$ of an object $X$.

\begin{lemma}\label{prop:heart}
  Let $v$ be a simple root and $X \in \heart$ be any object.
  The twist $\sigma^{-1}_{P_v}X$ lies in $\heart$ if and only if $P_v$ is not a sub-object of $X$.
  Similarly, the twist $\sigma_{P_v}X$ lies in $\heart$ if and only if $P_v$ is not a quotient of $X$.
\end{lemma}
\begin{proof}
  We prove the first statement; the second is analogous.
  Set $P = P_v$.
  We have the exact triangle
  \begin{equation}\label{eqn:simple-negative-twist}
    P \otimes \HHom(X,P)^\vee[-1] \to \sigma^{-1}_P(X) \to X \xrightarrow{+1}.
  \end{equation}
  Since both $P$ and $X$ lie in $\heart$, the 2-Calabi--Yau property implies that $\HHom^i(X,P)$ is zero for $i < 0$ and $i > 2$.

  Assume that $P$ is not a sub of $X$ in $\heart$.
  Since $P$ is simple in $\mathcal{A}$, we must have $\HHom^0(P,X) = 0$.
  By the 2-Calabi--Yau property, this implies $\HHom^2(X,P) = 0$.

  Let $V = \HHom(X,P)^\vee[-1]$.
  Since $V$ is a complex of vector spaces, it is quasi-isomorphic to a direct sum of shifted copies of $\k$.
  Since $\HHom^i(X,P) = 0$ if $i \notin \{0,1\}$, the complex $V$ must be a direct sum of copies of $\k[j]$ for $j = -1, 0$.
  
  In the exact triangle \eqref{eqn:simple-negative-twist}, the two extreme terms lie in $\mathcal{C}_{\geq -1} \cap \mathcal{C}_{< 1}$.
  As a result, $\sigma^{-1}_P(X)_{< 0}$ also lies in $\mathcal{C}_{\geq -1} \cap \mathcal C_{< 1}$.
  We must show that it lies in $\heart = \mathcal{C}_{\geq 0} \cap \mathcal{C}_{< 1}$.
  That is, we must show that its truncation to $\mathcal{C}_{< 0}$ is zero.

  Note that the truncation $\sigma^{-1}_P(X)_{< 0}$ lies in $\heart[-1]$ and coincides with $H^1(\sigma^{-1}_P(X))[-1]$.
  Since $H^1(X) = 0$, the cohomology long exact sequence applied to the triangle \eqref{eqn:simple-negative-twist} shows that $\sigma^{-1}_P(X)_{< 0}$   is a quotient of a direct sum of copies of $P[-1]$ in $\heart[-1]$.
  Since $P$ is simple in $\heart$, the object $\sigma^{-1}_P(X)_{< 0}$ must itself be a direct sum of copies of $P[-1]$.

  If $\sigma^{-1}_P(X)_{< 0}$ were non-zero, then we would have a non-zero map $\sigma^{-1}_P(X)_{< 0} \to P[-1]$, and hence a non-zero map $\sigma^{-1}_P(X) \to P[-1]$.
  By applying $\sigma_P$, we would then obtain a non-zero map $X \to P[-2]$, which is a contradiction.
  We conclude that  $\sigma_P^{-1}X$ lies in $\heart$.
  
  Conversely, if $P$ is a sub of $X$, then we have a non-zero map $P \to X$ and hence a non-zero map $\sigma^{-1}_P P = P[1] \to \sigma^{-1}_PX$.
  It follows that $\sigma^{-1}_PX$ is not in $\heart$.
\end{proof}

Consider an object $X \in \heart$.
We say that a subset $S \subset K$ \emph{envelops} the subs (resp. quotients) of $X$ if for every sub (resp. quotient) object $Y$ of $X$, the class $[Y]$ can be expressed as non-negative linear combination of the elements of $S$ and $\pm [X]$.
Observe that if $S$ envelops the subs of $X$ then $-S$ envelops the quotients of $X$, and vice-versa.

\begin{lemma}
  Let $X$ be an object of $\heart$ and let $v \in K$ be a simple root.
  Let $S$ be a subset of $K$ and set $S' = s_v(S) \cup \{v\}$.
  \begin{enumerate}
  \item If $S$ envelops the subs of $X$ and $\sigma^{-1}_{P_v}X$ lies in $\heart$, then $S'$ envelops the subs of $\sigma^{-1}_{P_v}X$.
  \item If $S \subset K$ envelops the quotients of $X$ and $\sigma_{P_v}X$ lies in $\heart$, then $S'$ envelops the quotients of $\sigma_{P_v}X$.
  \end{enumerate}
\end{lemma}
\begin{proof}
  We prove the first assertion; the second is similar.

  Set $P = P_v$.
  Let $Y$ be any sub of $\sigma^{-1}_{P}X$.
  We must prove that $[Y]$ is a non-negative linear combination of the elements of $S'$ and $\pm[\sigma_P^{-1}(X)]$.

  First suppose $P$ is not a quotient of $Y$.
  Set
  \[ Q = \coker(Y \to \sigma^{-1}_PX).\]
  Since $\sigma_P\sigma^{-1}_PX$ lies in $\heart$, by \autoref{prop:heart}, $P$ is not a quotient of $\sigma^{-1}_PX$.
  Therefore $P$ is not a quotient of $Q$.
  By applying $\sigma_P$ to the exact sequence
  \[ 0 \to Y \to \sigma^{-1}_PX \to Q \to 0,\]
  we get an exact triangle
  \[ \sigma_P Y \to X \to \sigma_PQ \xrightarrow{+1},\]
  whose terms are in $\heart$ by \autoref{prop:heart}.
  Therefore, it is an exact sequence in $\heart$.
  Since $S$ envelops the subs of $X$, the class $[\sigma_P Y] = s_v[Y]$ is a non-negative linear combination of elements of $S$ and $\pm [X]$.
  Equivalently, the class $[Y]$ is a non-negative linear combination of the elements of $s_v(S)$ and $\pm [\sigma_P^{-1}X]$.

  It remains to treat the case when $P$ is a quotient of $Y$.
  Define $Y' \subset Y$ be such that we have an exact sequence
  \[ 0 \to Y' \to Y \to P^{\oplus n} \to 0\]
  for some $n$ and $P$ is not a quotient of $Y$ (such a $Y'$ exists because $\heart$ is a finite-length category).
  By the argument above, $[Y']$ is a non-negative linear combination of the elements of $s_v(S)$ and $\pm [\sigma_P^{-1}X]$.
  But then $[Y]$ is a non-negative linear combination of the elements of $s_v(S) \cup \{v\}$ and $\pm [\sigma_P^{-1}X]$, as desired.  
\end{proof}

Consider an object $X$ of $\mathcal C$ defined as in \eqref{eqn:lift}:
\[
  X =  \sigma^{\epsilon_n}_{v_{n}} \circ \dots \circ \sigma_{v_{1}}^{\epsilon_1} (P_{v}),
\]
and the root sequence $R_i$, divided into a positive sub-sequence $R_+$, a negative sub-sequence $R-$, and the neutral root $R_0$.
\begin{lemma}\label{prop:hyperplane-heart}
  In the above setup, suppose there exists a linear functional $\lambda \from K(\mathcal C) \to {\mathbb R}$ such that $\lambda(R_0) = 0$, and $\lambda(R_+) \subset \mathbb R_{> 0}$, and $\lambda(R_-) \subset \mathbb R_{< 0}$.
  Then $X$ lies in the heart $\heart$.
  Furthermore, the set $R_-$ (resp. $R_+$) envelops the subs (resp. quotients) of $X$.
\end{lemma}
\begin{proof}
  We induct on $n$.
  If $n = 0$, then $X = P_v$ is simple, both $R_+$ and $R_-$ are empty, and the statement holds.

  Assume the statement for $(n-1)$.
  Let
  \[ X' = \sigma^{\epsilon_{n-1}}_{v_{n-1}} \circ \cdots \circ \sigma_{v_1}^{\epsilon_1} (P_v),\]
  and let $R'$ denote the root sequence for $X'$.
  Then we have $R'_i = s_{v_n} R_i$ for $i = 0, \dots, n-1$.
  Note that
  \[\lambda' = \lambda \circ s_{v_n} \from K(\mathcal C) \to \mathbb R\]
  is a linear functional that vanishes on the neutral root $R'_0$ for $X'$ and takes opposite signs on the positive and the negative sub-sequences $R'_+$ and $R'_-$.
  By the induction hypothesis, $X'$ lies in the heart $\heart$ and its subs (resp. quotients) are enveloped by $R'_-$ (resp. $R'_+$).

  Observe that $R_n = v_n$.
  Suppose $\epsilon_n = -1$.
  Then $R_n \in R_-$.
  Since $\lambda$ is negative on $R_-$, it is positive on $-R_n = s_{v_n}(R_n)$.
  By construction, $\lambda'$ is negative on $R'_-$, positive on $R_n$, and zero on $[X']$.
  So $R_n = [P_{v_n}]$ cannot be a positive linear combination of elements of $R'_-$ together with $\pm[X']$.
  
  Since $R'_-$ envelops the subs of $X'$, we conclude that $P_{v_n}$ is not a sub of $X'$.
  Hence, by \autoref{prop:heart}, $X$ is in the heart and its subs are enveloped by $s_1(R'_s) \cup \{v_n\} = R_-$.
  The proof when $\epsilon_n = +1$ is similar.  
\end{proof}

We now have the tools to finish the proof of \autoref{prop:unique-spherical-stable-construction}.
\begin{proof}[Proof of \autoref{prop:unique-spherical-stable-construction}]
Once again, set
\[
  X =  \sigma^{\epsilon_n}_{v_{n}} \circ \dots \circ \sigma_{v_{1}}^{\epsilon_1} (P_{v}),
\]
with $R_+$ and $R_-$ defined as above.
Suppose $Z$ maps $R_+$ and $R_-$ to $\mathbb H_+$ and $\mathbb H_-$, respectively.
Choose a linear functional $\ell \from \mathbb C \to \mathbb R$ that vanishes on $\alpha = Z(R_0)$ and takes positive (resp. negative) values on $\mathbb H_+$ (resp. $\mathbb H_-$).
Set $\lambda = \ell \circ Z$.
Then $\lambda$ satisfies the hypotheses of \autoref{prop:hyperplane-heart}.
As a result, $X$ is in the heart $\heart$.

To show that $X$ is semi-stable, consider a sub $Y \subset X$.
But $R_-$ envelops the subs of $X$, that is, $[Y]$ is a non-negative linear combination of elements of $R_-$ and $\pm [X]$.
By applying $Z$, we obtain that $Z(Y)$ is a non-negative linear combination of elements of $Z(R_-)$ and $\pm \alpha$.
Note that $Z(R_-) \subset \mathbb H_-$.
Since we know that $Z(Y) \subset \mathbb H$, we conclude that $Z(Y)$ lies in $\mathbb H_- \cup \mathbb R_{> 0} \cdot \alpha$.
As a result, we have $\arg Y \leq \arg X$.
Since this is true for any sub $Y \subset X$, we conclude that $X$ is semi-stable.
\end{proof}

\section{Applications}\label{sec:applications}
In this section, we reap the benefits of the results proved in the previous sections.

Fix the following notation:
\begin{itemize}
\item[$\mathcal C$] a $\k$-linear 2-Calabi--Yau triangulated category,
\item[$\tau$] a stability condition on $\mathcal C$,
\item[$\Phi$] the subset of $\mathbb R$ consisting of the phases of $\tau$-stable spherical objects
\item[$G$] the group of auto-equivalences of $\mathcal C$ generated by the twists in $\tau$-stable spherical objects.
\end{itemize}

\begin{proposition}\label{prop:one-orbit}
  With the notation introduced at the beginning of \autoref{sec:applications}, assume that $\tau$ admits at most one spherical stable object of every phase and $\Phi \subset \mathbb R$ is discrete.
  Then every spherical object in $\mathcal C$ is in the $G$-orbit of a $\tau$-stable spherical object.
\end{proposition}
\begin{proof}
  Let $Y$ be any spherical object of $\calC$.
  We denote by $|Y|$ the \emph{spread} of $Y$, which is the quantity
  \[ |Y| = \phi^+(Y) - \phi^-(Y).\]
  Note that since the stable factors of $Y$ must be spherical \cite[Corollary~2.3]{huy:12}, both $\phi^+(Y)$ and $\phi^-(Y)$ lie in $\Phi$, and hence their difference lies in the set $\{a-b \mid a, b \in \Phi, a \geq b\}$, which is a discrete subset of $\mathbb{R}$.

  We induct on $|Y|$.
  If $|Y| = 0$, then $Y$ is $\tau$-stable, and we are done.

  Otherwise, let $X$ be the unique $\tau$-stable spherical object of phase $\phi^-(Y)$.
  Since $Y$ is spherical and not stable, $X$ is not a direct summand of $Y$.
  Consider $Y' = \sigma^{-1}_X Y$.
  By \autoref{prop:improvement1} along with either \autoref{prop:improvement2} or \autoref{prop:improvement3}, we have  $|Y'| < |Y|$.
  By the induction hypothesis, $Y'$ lies in the $G$ orbit of $P$, and hence so does $Y$.
  
  Note that the same argument works with $Y' = \sigma_Z Y$ where $Z$ is the unique $\tau$-stable spherical object of phase $\phi^+(Y)$.
\end{proof}

Now suppose that $\Gamma$ is a finite quiver.
Recall that the braid group of $\Gamma$ acts on the associated 2-Calabi--Yau category via spherical twists.
\begin{corollary}[\autoref{thm:spherical}]\label{cor:spherical}
  Let $\mathcal C$ be the 2-Calabi--Yau category associated to a quiver of finite (ADE) type.
  Then every spherical object of $\mathcal C$ is in the braid group orbit of a simple object of the standard heart.
\end{corollary}
\begin{proof}
  Choose a stability condition $\tau$ on $\calC$ with the standard heart and generic central charge $Z \from K(\mathcal C) \to \mathbb C$.
  In particular, assume that $Z$ maps distinct roots to complex numbers of distinct arguments.
  Then, by \autoref{prop:unique-spherical-stable}, there is at most one spherical stable object of every phase.
  
  Let $\Phi \subset \mathbb R$ be the set of phases of spherical stable objects.
  Since the class of a spherical object in $K(\mathcal C)$ is a root, of which there are only finitely many, the set $\Phi$ consists of integer translates of a finite set.
  In particular, $\Phi$ is discrete.

  The simple objects of the standard heart are $\tau$-stable and spherical.
  So the image of the braid group lies in $G$.
  It is not hard to see that the image is, in fact, $G$.
  To see this, let $X$ be a $\tau$-stable spherical object.
  We must show that $\sigma_X$ lies in the image of the braid group.
  From \autoref{prop:unique-spherical-stable-construction}, we know that $X = \beta Y$, where $Y$ is a simple object in the standard heart, and $\beta$ is in the image of the braid group.
  Then
  \[\sigma_X = \beta \sigma_Y \beta^{-1}\]
  is also in the image of the braid group.

  We now apply \autoref{prop:one-orbit} and conclude that every spherical object is in the braid group orbit of a $\tau$-stable spherical object.
  But we already know that every $\tau$-stable spherical object is in the braid group orbit of a simple object of the standard heart.
  The proof is now complete.
\end{proof}

\begin{remark}[Choice of writing]
  Note that in \autoref{prop:one-orbit}, we have a choice of applying a positive or a negative twist, leading to different expressions for a given spherical object as a braid image of a simple object.
  It is an interesting problem to understand these different expressions.
\end{remark}

\begin{proposition}\label{prop:onestab}
  With the notation introduced at the beginning of \autoref{sec:applications}, assume that $\tau$ admits at most one spherical stable object of every phase and $\Phi \subset \mathbb R$ is discrete.
  Let $Y$ be a direct sum of spherical objects of $\mathcal C$ such that $\Hom^i(Y, Y) = 0$ for $i < 0$.
  Then there exists a stability condition $\omega$ in the $G$-orbit of $\tau$ such that $Y$ lies in the $[\alpha, \alpha+1)$-heart of $\omega$ for some $\alpha$.
\end{proposition}
\begin{proof}
  We induct on the spread $|Y| = |Y|_\tau$.
  Note that this quantity lies in the discrete set $\{a-b \mid a, b \in \Phi\}$.

  If $|Y|_\tau < 1$, then we are done.
  Simply take $\omega = \tau$ and $\alpha = \phi^-(Y)$.

  Suppose $|Y|_\tau \geq 1$.
  Let $X$ be the unique spherical $\tau$-stable object of phase $\phi^-(Y)$.
  Let $\tau' = \sigma_X \tau$.
  Note that the group $G$ and the set $\Phi$ are unchanged if we replace $\tau$ by $\tau'$.
  By \autoref{prop:improvement1} and \autoref{prop:improvement2}, we have
  \[ |Y|_{\tau'} = |\sigma_X^{-1}Y|_\tau < |Y|_\tau.\]
  We conclude the result by the induction hypothesis.
\end{proof}

\begin{corollary}[\autoref{thm:stability}]\label{cor:stability}
  Let $\mathcal C$ be the 2-Calabi--Yau category associated to a quiver of finite (ADE) type.
  Then the stability manifold of $\mathcal C$ is connected.
  Furthermore, up to rotation, every stability condition is in the braid group orbit of a standard stability condition.
\end{corollary}
\begin{proof}
  Let $\tau$ be an arbitrary stability condition on $\mathcal C$.
  Let $Z \from K(\calC) \to \mathbb C$ be its central charge.
  Perturb $\tau$ so that $Z$ maps distinct roots to complex numbers of distinct arguments.
  Note that the perturbed $\tau$ lies in the same connected component of the stability manifold as the original $\tau$.
  There is now a unique spherical $\tau$-stable object of every phase.

  We have shown in \autoref{prop:one-orbit} that every spherical object is in the braid group orbit of the simple objects of the standard heart.
  Therefore the subgroup $G$ of $\Aut(\mathcal C)$ generated by twists in $\tau$-stable spherical objects is a subgroup of the image of the braid group in $\Aut(\mathcal C)$.

  Let $\Phi \subset \mathbb R$ be the set of phases of spherical $\tau$-stable objects.
  Since there are finitely many roots, $\Phi$ consists of integer translates of a finite set, and hence is discrete.

  Let $Y$ be the direct sum of the simple objects in the standard heart of $\calC$.
  Note that $Y$ satisfies the hypotheses of \autoref{prop:onestab}.
  By \autoref{prop:onestab}, there is a stability condition $\omega$ in the braid group orbit of $\tau$ such that $Y$ is in the $[\alpha, \alpha+1)$ heart of $\omega$.
  Let $\omega'$ be the rotation of $\omega$ by $\alpha$, so that $Y$ lies in the $[0,1)$ heart of $\omega'$.
  Note that $\omega'$ is in the same connected component as $\omega$.

  The direct summands of $Y$ generate the standard heart of $\mathcal C$ (under taking extensions).
  Therefore, the $[0,1)$ heart of $\omega'$ contains the standard heart.
  Since both hearts are hearts of a $t$-structure, they must in fact be equal.
  In other words, $\omega'$ is a standard stability condition.

  Let $\Stab_0 \mathcal C$ be the connected component of $\Stab \mathcal C$ that contains the standard stability conditions.
  We have shown that an arbitrary $\tau \in \Stab \mathcal C$ is in the braid group orbit of a stability condition $\omega$ in $\Stab_0 \mathcal C$.
  But we know that the braid group preserves the connected component $\Stab_0 \mathcal C$ and in fact every $\tau \in \Stab_0$ is, up to rotation, in the braid group orbit of a standard stability condition \cite[\S~4]{ike:14}.
  Hence, we conclude that $\Stab \mathcal C = \Stab_0 \mathcal C$ and that every stability condition is, up to rotation, in the braid group orbit of a standard one.
\end{proof}

\providecommand{\bysame}{\leavevmode\hbox to3em{\hrulefill}\thinspace}
\providecommand{\MR}{\relax\ifhmode\unskip\space\fi MR }
\providecommand{\MRhref}[2]{\href{http://www.ams.org/mathscinet-getitem?mr=#1}{#2}
}
\providecommand{\href}[2]{#2}

\bibliographystyle{amsplain}

\begin{thebibliography}{1}

\bibitem{ada.miz.yan:19}
Takahide Adachi, Yuya Mizuno, and Dong Yang, \emph{Discreteness of silting
  objects and {$t$}-structures in triangulated categories}, Proc. Lond. Math.
  Soc. (3) \textbf{118} (2019), no.~1, 1--42.

\bibitem{bap.deo.lic:20*1}
Asilata Bapat, Anand Deopurkar, and Anthony~M. Licata, \emph{{A {T}hurston
  compactification of the space of stability conditions}},  (2020),
  \url{https://arxiv.org/abs/2011.07908}.

\bibitem{huy:12}
Daniel Huybrechts, \emph{Stability conditions via spherical objects}, Math. Z.
  \textbf{271} (2012), no.~3-4, 1253--1270.

\bibitem{ike:14}
Akishi Ikeda, \emph{{Stability conditions for preprojective algebras and root
  systems of Kac-Moody Lie algebras}}, 2014.

\bibitem{ish.ued.ueh:10}
Akira Ishii, Kazushi Ueda, and Hokuto Uehara, \emph{Stability conditions on
  {$A_n$}-singularities}, J. Differential Geom. \textbf{84} (2010), no.~1,
  87--126.

\bibitem{ish.ueh:05}
Akira Ishii and Hokuto Uehara, \emph{Autoequivalences of derived categories on
  the minimal resolutions of {$A_n$}-singularities on surfaces}, J.
  Differential Geom. \textbf{71} (2005), no.~3, 385--435.

\bibitem{sei.tho:01}
Paul Seidel and Richard Thomas, \emph{Braid group actions on derived categories
  of coherent sheaves}, Duke Math. J. \textbf{108} (2001), no.~1, 37--108.

\end{thebibliography}
\end{document}